\documentclass[12pt, reqno]{amsart}
\usepackage{amsmath, amsthm, amscd, amsfonts, amssymb, graphicx, color, mathrsfs}
\usepackage[bookmarksnumbered, colorlinks, plainpages]{hyperref}
\usepackage[all]{xy} 
\usepackage{slashed}

\newtheorem{theorem}{Theorem}[section]
\newtheorem{lemma}[theorem]{Lemma}

\newtheorem{proposition}[theorem]{Proposition}
\newtheorem{corollary}[theorem]{Corollary}
\theoremstyle{definition}
\newtheorem{definition}[theorem]{Definition}

\theoremstyle{remark}
\newtheorem{remark}[theorem]{Remark}
\numberwithin{equation}{section}

\begin{document}
\setcounter{page}{1}

\title[$L^p$-$L^q$ estimates for  pseudo-differential operators  ]{$L^p$-$L^q$ estimates for subelliptic  pseudo-differential operators on compact Lie groups}

 \author[D. Cardona]{Duv\'an Cardona}
\address{
 Duv\'an Cardona:
  \endgraf
  Department of Mathematics: Analysis, Logic and Discrete Mathematics
  \endgraf
  Ghent University, Belgium
  \endgraf
  {\it E-mail address} {\rm duvan.cardonasanchez@ugent.be}
  }
 \author[J. Delgado]{Julio Delgado}
\address{
  Julio Delgado:
  \endgraf
  Departmento de Matematicas
  \endgraf
  Universidad del Valle
  \endgraf
  Cali-Colombia
  \endgraf
    {\it E-mail address} {\rm delgado.julio@correounivalle.edu.co}}
  \author[V. Kumar]{Vishvesh Kumar}
\address{
 Vishvesh Kumar:
  \endgraf
  Department of Mathematics: Analysis, Logic and Discrete Mathematics
  \endgraf
  Ghent University, Belgium
  \endgraf
  {\it E-mail address} {\rm vishveshmishra@gmail.com, Vishvesh.Kumar@UGent.be}
  }

\author[M. Ruzhansky]{Michael Ruzhansky}
\address{
  Michael Ruzhansky:
  \endgraf
  Department of Mathematics: Analysis, Logic and Discrete Mathematics
  \endgraf
  Ghent University, Belgium
  \endgraf
 and
  \endgraf
  School of Mathematical Sciences
  \endgraf
  Queen Mary University of London
  \endgraf
  United Kingdom
  \endgraf
  {\it E-mail address} {\rm michael.ruzhansky@ugent.be}
  }

 \allowdisplaybreaks

\subjclass[2010]{Primary {22E30; Secondary 58J40}.}

\keywords{Pseudo-differential operator, Graded Lie groups, Mapping properties, Bessel potential, Rockland operators}

\thanks{The authors are supported  by the FWO  Odysseus  1  grant  G.0H94.18N:  Analysis  and  Partial Differential Equations and by the Methusalem programme of the Ghent University Special Research Fund (BOF)
(Grant number 01M01021). Vishvesh Kumar and Michael Ruzhansky are supported by  FWO Senior Research Grant G011522N. Duv\'an Cardona has been supported by the FWO Fellowship
grant No 1204824N.  Michael Ruzhansky is also supported  by EPSRC grants 
EP/R003025/2 and EP/V005529/1.}

\begin{abstract} We establish the $L^p$-$L^q$-boundedness of subelliptic pseudo-differential operators on a compact Lie group $G$. Effectively, we deal with the  $L^p$-$L^q$-bounds for operators  in the sub-Riemmanian setting because the subelliptic classes are associated to a H\"ormander sub-Laplacian.  The Riemannian case associated with the Laplacian is also included as a special case.   Then, applications to the $L^p$-$L^q$-boundedness of pseudo-differential operators in the H\"ormander classes on $G$ are given in the complete range $0\leq \delta\leq \rho\leq 1,$ $\delta<1.$  This also gives the $L^p$-$L^q$-bounds in the  Riemannian setting, because the later classes are associated with the Laplacian on $G$. In both cases, in the Riemannian and the sub-Riemannian settings, necessary and sufficient conditions for the $L^p$-$L^q$-boundedness of operators are also anaysed.

\end{abstract} \maketitle

\tableofcontents
\section{Introduction} 
This paper is mainly concerned with the $L^p$-$L^q$ boundedness of pseudo-differential operators associated with the global H\"ormander symbol classes on compact Lie groups for the range $1<p, q <\infty.$
Our analysis also includes estimates for pseudo-differential operators associated with  subelliptic symbol classes.

The relevance of the boundedness of Fourier multipliers and pseudo-differential operators has been highlighted by Stein and H\"ormander. These kinds of estimates naturally arise in the study of  some evolution equations. For instance, one can see \cite{Hor67,Stein71}. Till now, there have been extensive activities dealing with the $L^p$-$L^q$ boundedness for spectral multipliers and Fourier multipliers on compact Lie groups, we cite \cite{ARN,AR,CR20,Nikolskii,CK1,KR,KRIMRN,RR23} for a non-exhaustive list of references. 

To the best of our knowledge, there has  been no activity to pursue the $L^p$-$L^q$ estimates of non-invariant operators, in particular, pseudo-differential operators on compact Lie groups. In the classical Euclidean setting, H\"ormander established the $L^p$-$L^q$ estimates of pseudo-differential operators associated with the so-called ``H\"ormander symbol classes" $S^m_{\rho, \delta}(\mathbb{R}^n \times \mathbb{R}^n)$ on $\mathbb{R}^n$ with $m \in \mathbb{R}$ and $0 \leq \delta <\rho \leq 1.$ It is well-known that any $L^p$-$L^q$ bounded Fourier multiplier is nontrivial only if $p \leq q$ (see \cite{Hormander1960}). Therefore, it is natural to assume the condition $p \leq q$ when dealing with pseudo-differential operators. Later, \'Alvarez and Hounie \cite{Hounie} extended H\"ormander's result to the range $0 \leq \delta <1$ and $0<\rho \leq 1$ without the restriction $\delta < \rho.$ In a recent work by the first and last two authors \cite{LpLqGraded}, we have provided sufficient and necessary conditions for the $L^p$-$L^q$ boundedness of pseudo-differential operators associated with global H\"ormander symbol classes $S^m_{\rho, \delta}(G \times \widehat{G}), \,m \in \mathbb{R},  0\leq \delta\leq \rho\leq 1$ and $\delta\neq 1,$ on a graded Lie group $G,$ where $\widehat{G}$ denotes the  unitary dual of $G.$

In this work, we focus on pseudo-differential operators associated with the global H\"ormander symbol classes encoded with the Riemannian and sub-Riemannian structure of compact Lie groups. One of the main differences between the approach developed in \cite{LpLqGraded}, based on the analysis of hypoelliptic operators on those groups, is the use of the structure of the dilations of the group, while the approach of this paper will be based on the sub-markovian properties of the semigroup $e^{-t\mathcal{L}},$ $t>0,$ of a H\"ormander sub-Laplacian $\mathcal{L}=\sum_{j=1}^kX_j^2,$ where one exploits the geometric properties induced by the H\"ormander system of vector fields $\{X_j:1\leq j\leq n\}$ on a compact Lie group $G.$  

On a compact Lie group $G$, in the monograph \cite{Ruz}, Turunen and the last author introduced a global notion of the H\"ormander symbol classes on $G$. According to this terminology, and observing that any continuous linear operator $A$ acting on $C^\infty(G)$ has a right convolution kernel $R_A=R_{A}(x,y)\in \mathscr{D}'(G\times G),$ namely, a distribution that describes the action of the operator by the group convolution $*$ as follows
\begin{equation}
    Af(x)=(f\ast R_{A}(x,\cdot))(x),
\end{equation} the {\it global symbol} of $A,$ is the matrix-valued function defined  on $G\times \widehat{G},$
 defined via
\begin{equation}
    \sigma_A(x,\xi)=\widehat{R}_A(x,\xi),\,(x,[\xi])\in G\times \widehat{G}.
\end{equation}Here, $\widehat{\cdot}$ denotes the matrix-valued group Fourier transform on $G$. By classifying these matrix-valued symbols by the behaviour of their derivatives (and of their differences), the last author  and Turunen \cite{Ruz} introduced the symbols classes $S^{m}_{\rho,\delta}(G\times \widehat{G}),$ allowing the complete range $0\leq \delta\leq \rho\leq 1,$ and providing a new description of the H\"ormander classes $S^{m}_{\rho,\delta}(T^{*}G)$ (as defined in \cite{HormanderBook34} with the local notion of the principal symbol, defined on the cotangent-bundle $T^{*}M$ of a compact manifold) when additionally, $0\leq \delta<\rho\leq 1,$ and $\rho>1-\delta.$

On the other hand, it was observed by the first and the last author in \cite{CR20}, that the symbols classes $S^{m}_{\rho,\delta}(G\times \widehat{G}),$ are associated to the Riemannian structure of the group $G$, in the sense that the growth of the derivatives of  symbols is classified according in terms of the spectrum of the Laplacian $\mathcal{L}_G=-\sum_{j=1}^{n}X_j^2,$ $n=\dim(G).$ Then, in generalising this idea, in   \cite{CR20} the subelliptic H\"ormander classes $S^{m,\mathcal{L}}_{\rho,\delta}(G\times \widehat{G}),$ were introduced with the derivatives (and differences) of symbols compared with respect to the growth of the eigenvalues of a fixed H\"ormander sub-Laplacian $\mathcal{L}=-\sum_{j=1}^kX_j^2,$ where $k<n.$ We observe that the pseudo-differential calculus  associated to the ``subelliptic'' classes $S^{m,\mathcal{L}}_{\rho,\delta}(G\times \widehat{G}),$ is more singular than the one associated to the ``elliptic'' classes $S^{m}_{\rho,\delta}(G\times \widehat{G}).$ Indeed, singularities of the kernels of the ``subelliptic'' classes are classified in terms of the Hausdorff dimension $Q,$ associated with the control distance associated with the sub-Laplacian  $\mathcal{L}.$ In the next subsection we present the $L^p-L^q$ regularity properties of the subelliptic H\"ormander classes $S^{m,\mathcal{L}}_{\rho,\delta}(G\times \widehat{G}).$   

We finally observe that for the case $p=q,$ namely, the problem regarding  the $L^p$-boundedness of pseudo-differential operators, Fefferman in \cite{F73} has established a sharp criterion of continuity for the operator in the H\"ormander classes $S^{m}_{\rho,\delta}(\mathbb{R}^n\times \mathbb{R}^n)$ on the Euclidean space. Then Fefferman's criterion has been extended for several pseudo-differential calculi including the Weyl-H\"ormander calculus \cite{Delgado2006}, the H\"ormander classes $S^{m}_{\rho,\delta}(G\times \widehat{G})$ associated to the Laplacian \cite{DR19}, also extended in the sub-Riemannian setting, namely, for the H\"ormander classes   $S^{m,\mathcal{L}}_{\rho,\delta}(G\times \widehat{G})$  associated to a H\"ormander sub-Laplacian \cite{CR20}, and finally for the H\"ormander classes on graded Lie groups in \cite{CDR21JGA}. In order to give a general perspective about this problem, here we are mainly concerned with the case $p<q.$ 

Notably, when dealing with the $L^p$-$L^q$-boundedness of operators with symbols in the classes $S^{m,\mathcal{L}}_{\rho,\delta}(G\times \widehat{G}),$ one has to analyse separately the cases: (i) $1<p\leq q\leq 2,$ (ii) $1<p\leq 2\leq q<\infty,$ and (iii) $2\leq p\leq q<\infty.$ Here for the case $1<p\leq 2\leq q<\infty$ we provide necessary and sufficient conditions.
\subsection{Main results}

The following theorem presents the result that establishes a sufficient condition and in some cases also a necessary condition, for the $L^p$-$L^q$ boundedness of subelliptic pseudo-differential operators on compact Lie groups.
\begin{theorem}\label{main:theorem} 
Let $1<p, q<\infty,$ and   $0\leq \delta< \rho\leq 1.$ Let $G$ be a compact Lie group, and let $Q$ be its Hausdorff dimension with respect to the control distance associated with a H\"ormander sub-Laplacian $\mathcal{L}.$  Then, the following statements hold.
\begin{itemize}
    \item Let $1<p\leq2\leq  q<\infty.$ Every pseudo-differential operator $A\in \Psi^{m,\mathcal{L}}_{\rho,\delta}(G\times \widehat{G})$ admits a bounded extension from $L^p(G)$ into $L^q(G),$ that is
\begin{equation}\label{Lp-Lq:bound}
   \forall f\in C^{\infty}(G),\, \Vert A f\Vert_{L^q(G)}\leq  C\Vert  f\Vert_{L^p(G)}\,\,\,
\end{equation} holds, if and only if, 
\begin{equation}\label{Necessary:condition:3:intro}
   m\leq -Q\left(\frac{1}{p}- \frac{1}{q}\right).
\end{equation}  
\item Every pseudo-differential operator $A\in \Psi^{m,\mathcal{L}}_{\rho,\delta}(G\times \widehat{G})$ admits a bounded extension from $L^p(G)$ into $L^q(G),$ that is \eqref{Lp-Lq:bound} holds, in the following cases: 
\begin{itemize}
        \item[(i)] if $1<p\leq q \leq 2$ and  \begin{equation}\label{Necessary:condition:4:intro}
   m\leq -Q \left( \frac{1}{p}-\frac{1}{q}+(1-\rho) \left(\frac{1}{q}-\frac{1}{2}\right)\right). \end{equation}    
\item[(ii)] if $2 \leq p \leq q<\infty$ and 

\begin{equation}\label{Necessary:condition:3:intro:}
   m\leq -Q\left( \frac{1}{p}-\frac{1}{q}+(1-\rho) \left(\frac{1}{2}-\frac{1}{p}\right)\right).
   \end{equation}
\end{itemize}
\end{itemize}
\end{theorem}

\begin{remark} \label{rem1}  The order conditions in \eqref{Necessary:condition:3:intro},  \eqref{Necessary:condition:4:intro} and \eqref{Necessary:condition:3:intro:} can be written in a simplified way for $1<p,q<\infty$ as follows:
\begin{equation}\label{order:condition}
    m\leq -Q\left(   \frac{1}{p}-\frac{1}{q}+(1-\rho)\max\left\{ \frac{1}{2}-\frac{1}{p},\frac{1}{q}-\frac{1}{2},0\right\}\right),
\end{equation}where $Q$ is the Hausdorff dimension of $G$ with respect to the control distance associated to the sub-Laplacian $\mathcal{L}.$    
\end{remark}
\begin{remark}
    If $G=\mathbb{R}^n,$ although this is not a compact Lie group,  the order condition in \eqref{order:condition} is sharp for Fourier multipliers, see H\"ormander \cite[Page 163]{Hor67}.
\end{remark}
\begin{remark}
    When a system of vector fields $X=\{X_j\}$ provides an orthonormal basis of the Lie algebra (endowed, up to a constant factor, with is unique bi-invariant Riemannian metric), the H\"ormander condition is trivially satisfied, the sub-Laplacian associated to the system $X$ coincides with the Laplacian and the classes $S^{m,\mathcal{L}}_{\rho,\delta}(G\times \widehat{G})$ agree with the ``elliptic  classes'' $S^{m}_{\rho,\delta}(G\times \widehat{G})$ of the last author with Turunen \cite{Ruz}. The following corollary provides the $L^p-L^q$-regularity properties for the elliptic classes.
\end{remark}
\begin{corollary}\label{main:cor} 
Let $1<p, q<\infty,$ and   $0\leq \delta< \rho\leq 1.$ Let $G$ be a compact Lie group of dimension $n.$  Then, the following statements hold.
\begin{itemize}
    \item Let $1<p\leq2\leq  q<\infty.$ Every pseudo-differential operator $A\in \Psi^{m}_{\rho,\delta}(G\times \widehat{G})$ admits a bounded extension from $L^p(G)$ into $L^q(G),$ that is
\begin{equation}\label{Lp-Lq:bound:cor}
   \forall f\in C^{\infty}(G),\, \Vert A f\Vert_{L^q(G)}\leq  C\Vert  f\Vert_{L^p(G)}\,\,\,
\end{equation} holds, if and only if, 
\begin{equation}\label{Necessary:condition:3:intro:cor}
   m\leq -n\left(\frac{1}{p}- \frac{1}{q}\right).
\end{equation}  
\item Every pseudo-differential operator $A\in \Psi^{m}_{\rho,\delta}(G\times \widehat{G})$ admits a bounded extension from $L^p(G)$ into $L^q(G),$ that is \eqref{Lp-Lq:bound:cor} holds, in the following cases: 
\begin{itemize}
        \item[(i)] if $1<p\leq q \leq 2$ and  \begin{equation}\label{Necessary:condition:4:intro:cor}
   m\leq -n \left( \frac{1}{p}-\frac{1}{q}+(1-\rho) \left(\frac{1}{q}-\frac{1}{2}\right)\right). \end{equation}    
\item[(ii)] if $2 \leq p \leq q<\infty$ and 

\begin{equation}\label{Necessary:condition:3:intro:cor:intro}
   m\leq -n\left( \frac{1}{p}-\frac{1}{q}+(1-\rho) \left(\frac{1}{2}-\frac{1}{p}\right)\right).
   \end{equation}
\end{itemize}
\end{itemize}
\end{corollary}
This paper is organised as follows. In Section \ref{sec2} we present the preliminaries of this paper related to sub-Markovian semigroups, and the subelliptic pseudo-differential calculus introduced in  \cite{CR20}. Subsequently, the $L^p-L^q$-boundedness of these subelliptic classes is analysed in Section  \ref{Section:lp:lq}. Then, in Section \ref{ex} we provide explicit   examples of our criterion about the $L^p$-$L^q$-boundedness of pseudo-differential operators in the case of the sphere $\mathbb{S}^3\cong \textnormal{SU}(2)$ and on $\textnormal{SU}(3).$

\section{Preliminaries} \label{sec2}

\subsection{Symmetric submarkovian semigroups} We briefly recall some classical facts concerning symmetric submarkovian 
semigroups on $L^2:=L^2(X,\mu).$ Here $(X,\mu)$ is a $\sigma$-finite measure space. For the definitions and results mentioned in this sub-section we follow \cite[Section II.5]{Saloff-CosteBook}, \cite[Example II.5.1]{Saloff-CosteBook} (see Remark \ref{Remark} below) and \cite[Theorem II.3.1, Page 14]{Saloff-CosteBook} (see Theorem \ref{Converse:HLT} below). 

Let $A$ be an operator with domain $\textnormal{Dom}(A)\subset L^2.$ We recall that

{\it $-A$ is the generator of a symmetric semigroup $T_{t}:=e^{-tA}$ on $L^2$ such that 
$$\Vert e^{-tA}\Vert_{L^2\rightarrow L^2}\leq e^{\alpha t}  $$
if and only if $A$ is self-adjoint, $\textnormal{Dom}(A)$ is a dense subspace of $L^2$, and $$(Af,f)\geq -\alpha\Vert f\Vert_{L^2}.$$  }

Let $Q$ be a symmetric bilinear form defined on a subspace $D\subset L^2.$ One says that $Q$ is positive if $Q(f,f)\geq 0,$ and closed if for every sequence $(f_n)_{n\in \mathbb{N}}\subset D,$ such that $f_{n}\rightarrow f$ in $L^2,$ and 
$$\lim_{n,m\rightarrow \infty}Q(f_n-f_m,f_{n}-f_{m})=0,$$
one has that $f\in D$ and that $Q(f_n-f,f_n-f)\rightarrow 0.$ One says that $Q$ is closable if it admits a closed extension.

If $A$ is a symmetric operator with a dense domain $\textnormal{Dom}(A)\subset L^2,$ then one may associate with it the symmetric bilinear form $Q_A(f,g):=Q(Af,g).$ If in addition $Q_A$ is positive, it is closable and its minimal closure $\overline{Q}_A$ is associated to a self-adjoint operator $\overline{A}$ which is an extension of $A.$ More precisely, $\overline{A}$ is the smallest self-adjoint extension of A, called the {\it Friedrich extension  } of $A.$ We shall not distinguish between $A$ and $\overline{A}.$

Recall that a semigroup $T_{t}$ on $L^2$ is called {\it submarkovian } if 
$f\in L^2,$  $0\leq f\leq 1,$ implies that $0\leq T_tf\leq 1.$ 
Such a semigroup acts on the $L^p$-spaces and $\Vert T_t\Vert_{L^p\rightarrow L^p}\leq 1.$

Symmetric sub-markovian semi-groups on $L^2$ may be characterised through properties of the associated symmetric bilinear form. A positive symmetric bilinear form $Q$ defined on $D\subset L^2,$ is said to be a {\it Dirichlet form}, if for all $g\in D,$ and for all $f\in D,$ such that $|f|\leq |g|,$ and  $|f(x)-f(y)|\leq |g(x)-g(y)|,$  one has that $f\in D,$ and $Q(f,f)\leq Q(g,g).$ 

If $T_t=e^{-tA}$ is a symmetric sub-markovian semigroup on $L^2,$ the associated bilinear form $Q(f,g):=(A^{1/2}f,A^{1/2}g),$ $f,g\in \textnormal{Dom}(A^{1/2}),$ is a closed 
Dirichlet form with dense domain in $L^2.$ Conversely, given a closed Dirichlet bilinear form $Q,$ with dense domain $D$ in $L^2,$ there exists a unique symmetric sub-markovian semigroup on $L^2,$ such that $T_t=e^{-tA}$ and $Q(f,g):=(A^{1/2}f,A^{1/2}g),$ $f,g\in \textnormal{Dom}(A^{1/2}).$
\begin{remark}\label{Remark}
    We note that sub-markovian semigroups arise naturally in the setting of compact manifolds. Indeed, if $M$ is a compact manifold with a volume form $dx,$ and $L^2(M)=L^2(M,dx),$ consider a family of vector fields
    $$X=\{X_1,X_2,\cdots, X_k\},\,k\leq n:=\dim(M).$$ If every $X_j$ is skew-adjoint on $L^2(M),$ namely, if 
    $$\forall f,g\in C^{\infty}(M),\,\, \smallint_{M}X_i(f)gdx=-\smallint_{M}fX_i(g)dx, $$ then we can associate with $\Delta_X:=-\sum_{i=1}^kX_i^2,$ its Friedrichs extension, which we still denote by $\Delta_X.$ Then, the semigroup
    $$T_{t}:=e^{-t\Delta_X}:L^2(M)\rightarrow L^{2}(M)$$ is a contraction semigroup. With respect to the symmetric bilinear form
$$Q(f,g)=(\Delta_X f,g),\,f,g\in C^{\infty}(M),$$ the semigroup $T_{t}:=e^{-t\Delta_X}$ is a sub-markovian semigroup. 
\end{remark} The following theorem will be fundamental for our further analysis.
\begin{theorem}\label{Converse:HLT} Let $T_{t}=e^{-tA}$ be a sub-markovian semigroup. Assume that $T_{t}$ is equicontinuous on $L^1(X,\mu)$ and on $L^{\infty}(X,\mu).$ Suppose that there exists $\alpha>0$ and $1<p<q\leq \infty$ such that 
\begin{equation}
    \Vert f\Vert_{L^q(X,\mu)}\leq C\Vert A^{\alpha/2} f\Vert_{L^p(X,\mu)}.
\end{equation} Then, with $Q$ defined by the identity $\alpha=Q(1/p-1/q),$ the following semigroup estimate holds
\begin{equation}
    \exists C>0,\forall f\in L^1(X,\mu),\,\,\Vert T_t f\Vert_{L^\infty(X,\mu)}\leq C t^{-\frac{Q}{2}}\Vert f\Vert_{L^\infty(X,\mu)}.
\end{equation}    
\end{theorem}
\subsection{Pseudo-differential operators via localisations}\label{S2.1} Now we present the preliminaries of the H\"ormander theory of
 pseudo-differential operators on compact manifolds used in this work. The setting of compact Lie groups appears as an essential case of manifolds with symmetries.  We refer to H\"ormander \cite{HormanderBook34} for details.\\
 
 Let $U$ be an open  subset of $\mathbb{R}^n.$ We  say that  the symbol $a\in C^{\infty}(T^{*}U),$ $T^{*}U=U\times \mathbb{R}^n,$  belongs to the H\"ormander class  $$S^m_{\rho,\delta}(T^{*}U),\,0\leqslant \rho,\delta\leqslant 1,$$ if for every compact subset $K\subset U$  the symbol inequalities
\begin{equation*}
  |\partial_{x}^\beta\partial_{\xi}^\alpha a(x,\xi)|\leqslant C_{\alpha,\beta,K}(1+|\xi|)^{m-\rho|\alpha|+\delta|\beta|},
\end{equation*} hold true uniformly in $x\in K$ for all  $\xi\in \mathbb{R}^n.$  A continuous linear operator $ A:C^\infty_0(U) \rightarrow C^\infty(U)$  
is a pseudo-differential operator of order $m$ and of  $(\rho,\delta)$-type, if there exists
a symbol $a\in S^m_{\rho,\delta}(T^{*}U)$ such that $A$ is the Kohn-Nirenberg quantisation of the symbol $a,$ namely, if
\begin{equation*}
    Af(x)=\smallint\limits_{\mathbb{R}^n}e^{2\pi i x\cdot \xi}a(x,\xi)(\mathscr{F}_{\mathbb{R}^n}{f})(\xi)d\xi,
\end{equation*} for all $f\in C^\infty_0(U),$ where
$$
    (\mathscr{F}_{\mathbb{R}^n}{f})(\xi):=\smallint\limits_Ue^{-i2\pi x\cdot \xi}f(x)dx
$$ is the  Euclidean Fourier transform of $f$ at $\xi\in \mathbb{R}^n.$ \\

Now, we extend this notion to smooth manifolds as follows.  Given a smooth closed manifold $M,$  $A:C^\infty_0(M)\rightarrow C^\infty(M) $ is a pseudo-differential operator of order $m$ and of $(\rho,\delta)$-type, with $ \rho\geqslant   1-\delta, $ and $0\leq \delta<\rho\leq 1,$  if for every local  coordinate system $\omega: M_{\omega}\subset M\rightarrow U_{\omega}\subset \mathbb{R}^n,$
and for every $\phi,\psi\in C^\infty_0(U_\omega),$ the operator
\begin{equation*}
    Tu:=\psi(\omega^{-1})^*A\omega^{*}(\phi u),\,\,u\in C^\infty(U_\omega),\footnote{As usually, $\omega^{*}$ and $(\omega^{-1})^*$ are the pullbacks, induced by the maps $\omega$ and $\omega^{-1}$ respectively.}
\end{equation*} is a standard pseudo-differential operator with symbol $a_T\in S^m_{\rho,\delta}(T^{*}U_\omega).$ In this case we write $A\in \Psi^m_{\rho,\delta}(M,\textnormal{loc}).$

\subsection{Positive sub-Laplacians and global pseudo-differential operators} Let $G$ be a compact Lie group with Lie algebra $\mathfrak{g}\simeq T_{e}G$, where $e$ is the neutral element of $G$, and let  
$$X=\{X_1,\cdots,X_{k} \}\subset \mathfrak{g}$$ 
be a system of $C^\infty$-vector fields. For all multi-index,
$$I=(i_1,\cdots,i_\omega)\in \{1,2,\cdots,k\}^{\omega}$$ of length $\omega\geqslant   1$, we denote by $$X_{I}:=[X_{i_1},[X_{i_2},\cdots [X_{i_{\omega-1}},X_{i_\omega}]\cdots]]$$
a commutator of length $\omega$, where $X_{I}:=X_{i}$ when $\omega=1$ and $I=(i)$. The system $X$ satisfies the  H\"ormander condition of step  $\kappa$ if $$\mathfrak{g}=\mathrm{span}\{X_I: |I|\leq \kappa\}.$$

Given a system $X=\{X_1,\cdots,X_{k}\}$ satisfying  the H\"ormander condition,  the operator defined as
\begin{equation*}
    \mathcal{L}\equiv \mathcal{L}_{X}:=-(X_{1}^2+\cdots +X_{k}^2),
\end{equation*} 
 is called the subelliptic Laplacian associated with the system $X$, or simply the sub-Laplacian associated to $X$. The subellipticity of $\mathcal{L}$ follows from the validity of the estimate, (see H\"ormander  \cite{Hormander1967} and Rothschild and Stein \cite{RS76}) 
\begin{equation}
    \Vert u\Vert_{H^s(G)}\leqslant C(\Vert \mathcal{L}u\Vert_{L^2(G)}+\Vert u\Vert_{L^2(G)}),
\end{equation} with $s=2/\kappa, $ while the Sobolev space $H^s$ of order $s$ is defined by the norm $$\Vert u\Vert_{H^s(G)}:=\Vert (1-\Delta)^{\frac{s}{2}} u\Vert_{L^2(G)}.$$ Here,  $\Delta$ is the negative Laplace-Beltrami operator on $G.$

Let us now introduce the Hausdorff dimension associated with the sub-Laplacian $\mathcal{L}$. For all $x\in G,$ let $H_{x}^\omega G$ be the linear subspace of $T_xG$ generated by the $X_i$'s and by the Lie brackets $$ [X_{j_1},X_{j_2}],[X_{j_1},[X_{j_2},X_{j_3}]],\cdots, [X_{j_1},[X_{j_2}, [X_{j_3},\cdots, X_{j_\omega}] ] ],$$ where  $\omega\leqslant \kappa.$ Then, H\"ormander's condition says that  $H_{x}^\kappa G=T_xG,$  $x\in G$, and we have that
\begin{equation*}
H_{x}^1G\subset H_{x}^2G \subset H_{x}^3G\subset \cdots \subset H_{x}^{\kappa-1}G\subset H_{x}^\kappa G= T_xG,\,\,x\in G.
\end{equation*} The dimension of every $H_x^\omega G$ is constant in $x\in G$, so we set $\dim H^\omega G:=\dim H_{x}^\omega G,$ for all $x\in G.$ The Hausdorff dimension can be defined as,
\begin{equation}\label{Hausdorff-dimension}
    Q:=\dim(H^1G)+\sum_{i=1}^{\kappa-1} (i+1)(\dim H^{i+1}G-\dim H^{i}G ).
\end{equation}

Let  $A$ be a continuous linear operator from $C^\infty(G)$ into $\mathscr{D}'(G),$ and let  $\widehat{G}$ be  the  unitary dual of $G.$  There exists a matrix-valued function \begin{equation}\label{symbol}a:G\times \widehat{G}\rightarrow \cup_{\ell\in \mathbb{N}} \mathbb{C}^{\ell\times \ell},\end{equation}  that we call the matrix symbol of $A,$ such that $a(x,\xi):=a(x,[\xi])\in \mathbb{C}^{d_\xi\times d_\xi}$ for every  $[\xi]\in \widehat{G},$ with $\xi:G\rightarrow \textnormal{Hom}(H_{\xi}),$ $H_{\xi}\cong \mathbb{C}^{d_\xi},$ and such that
\begin{equation}\label{RuzhanskyTurunenQuanti}
    Af(x)=\sum_{[\xi]\in \widehat{G}}d_\xi{\textnormal{Tr}}[\xi(x)a(x,\xi)\widehat{f}(\xi)],\,\,\forall f\in C^\infty(G).
\end{equation}We have denoted by
\begin{equation*}
    \widehat{f}(\xi)\equiv (\mathscr{F}f)(\xi):=\smallint\limits_{G}f(x)\xi(x)^*dx\in  \mathbb{C}^{d_\xi\times d_\xi},\,\,\,[\xi]\in \widehat{G},
\end{equation*} the group Fourier transform of $f$ at $\xi$ where the matrix representation of $\xi$ is induced by an orthonormal basis of the representation space $H_{\xi}.$
Correspondingly, one denotes the inverse Fourier transform of $g(\xi)\in \mathbb{C}^{d_\xi\times d_\xi}$ as
$$(\mathscr{F}^{-1}g)(x):=\sum_{[\xi]\in\widehat{G}}d_\xi \mathrm{Tr}(\xi(x)g(\xi)),\quad x\in G.$$
Note that the matrix-valued function $a$ in \eqref{symbol} satisfying \eqref{RuzhanskyTurunenQuanti} is unique, and satisfies the identity
\begin{equation*}
    a(x,\xi)=\xi(x)^{*}(A\xi)(x),\,\, A\xi:=(A\xi_{ij})_{i,j=1}^{d_\xi},\,\,\,[\xi]\in \widehat{G}.
\end{equation*}

We will use the notation $A=\textnormal{Op}(a)$ to indicate that $a:=\sigma_A(x,\xi)$ is the (unique) matrix-valued symbol associated with $A.$ 

 As defined in   \cite{RuzhanskyWirth2015}, a difference operator $Q_\xi: \mathscr{D}'(\widehat{G})\rightarrow \mathscr{D}'(\widehat{G})$ of order $k\in \mathbb{N}$  is defined via
\begin{equation}\label{taylordifferences}
    Q_\xi\widehat{f}(\xi)=\widehat{qf}(\xi),\,[\xi]\in \widehat{G},
\end{equation}  for some function $q\in C^\infty(G)$ vanishing of order $k$ at  $x=e.$ We  denote by $\textnormal{diff}^k(\widehat{G})$  the set of the difference operators of order $k.$ The associated difference operator to $q$ is denoted by $\Delta_q\equiv Q_\xi.$ A system  of difference operators (see  \cite{RuzhanskyWirth2015})
\begin{equation}\label{def.differenceoperators}
  \Delta_{\xi}^\alpha:=\Delta_{q_{(1)}}^{\alpha_1}\cdots   \Delta_{q_{(i)}}^{\alpha_{i}},\,\,\alpha=(\alpha_j)_{1\leqslant j\leqslant i}, 
\end{equation}
with $i\geq n$, is called  an admissible family,  if
\begin{equation}\label{def.admissible.diff}
    \textnormal{rank}\{\nabla q_{(j)}(e):1\leqslant j\leqslant i \}=\textnormal{dim}(G), \textnormal{   and   }\Delta_{q_{(j)}}\in \textnormal{diff}^{1}(\widehat{G}).
\end{equation}
An admissible family is said to be strongly admissible if, we also have the property 
\begin{equation}\label{def.stongly.admissible.diff}
    \bigcap_{j=1}^{i}\{x\in G: q_{(j)}(x)=0\}=\{e\}.
\end{equation}

\begin{remark}\label{remarkD}
We observe that matrix components of unitary irreducible representations induce difference operators of arbitrary order. Let us illustrate this fact as follows. If $\xi_{1},\xi_2,\cdots, \xi_{k},$ are  fixed irreducible and unitary  representations of the group $G$, which does not necessarily belong to the same equivalence class, then the matrix coefficients
\begin{equation}
 \xi_{\ell}(g)-I_{d_{\xi_{\ell}}}=[\xi_{\ell}(g)_{ij}-\delta_{ij}]_{i,j=1}^{d_{\xi_\ell}},\, \quad g\in G, \,\,1\leq \ell\leq k,
\end{equation} 
define the smooth functions 
$q^{\ell}_{ij}(g):=\xi_{\ell}(g)_{ij}-\delta_{ij}$, $ g\in G,$ and then  define the difference operators
\begin{equation}\label{Difference:op:rep}
    \mathbb{D}_{\xi_\ell,ij}:=\mathscr{F}(\xi_{\ell}(g)_{ij}-\delta_{ij})\mathscr{F}^{-1}.
\end{equation}
Then, by fixing $k\geq \mathrm{dim}(G)$ of these unitary representations with the property that its corresponding  family of difference operators is admissible one can define higher-order difference operators of this kind. Indeed,  let us fix a unitary representation $\xi_\ell$.
We omit the index $\ell.$ 
Then, for any given multi-index $\alpha\in \mathbb{N}_0^{d_{\xi_\ell}^2}$, with 
$|\alpha|=\sum_{i,j=1}^{d_{\xi_\ell}}\alpha_{ij}$, we write
$$\mathbb{D}^{\alpha}:=\mathbb{D}_{11}^{\alpha_{11}}\cdots \mathbb{D}^{\alpha_{d_{\xi_\ell}d_{\xi_\ell}}}_{d_{\xi_\ell}d_{\xi_\ell}}
$$ 
for a difference operator of order $|\alpha|$.
\end{remark}
The difference operators endow the unitary dual $\widehat{G}$ with a difference structure. Indeed, the  following Leibniz  formula holds true (see  \cite{RuzhanskyTurunenWirth2014} for details). We refer to Definition \ref{def.elliptic.classes} for the description via the group Fourier transform of the matrix-valued distributions in the class  $\mathscr{D}'(G\times \widehat{G}).$
\begin{proposition}[Leibniz rule for difference operators]\label{Leibnizrule} Let $G$ be a compact Lie group and let $\mathbb{D}^{\alpha},$ $\alpha\in \mathbb{N}^{d_{\xi_\ell}}_0,$ be the family of difference operators defined in  \eqref{Difference:op:rep}. Then, the following Leibniz rule 
\begin{align*}
    \mathbb{D}^{\alpha}(a_{1}a_{2})(x_0,\xi) =\sum_{ |\gamma|,|\varepsilon|\leqslant |\alpha|\leqslant |\gamma|+|\varepsilon| }C_{\varepsilon,\gamma}(\mathbb{D}^{\gamma}a_{1})(x_0,\xi) (\mathbb{D}^{\varepsilon}a_{2})(x_0,\xi), \quad x_{0}\in G,
\end{align*}holds  for all $a_{1},a_{2}\in \mathscr{D}'(G\times \widehat{G})$, where the summation is taken over all $\varepsilon, \gamma$ such that $|\varepsilon|,|\delta|\leq |\alpha|\leq |\gamma|+|\varepsilon|$.  
\end{proposition}

Now, we will introduce the  H\"ormander classes of matrix-symbols defined in \cite{Ruz}.
We identify every $Y\in\mathfrak{g}$ with the  differential operator $\partial_{Y}:C^\infty(G)\rightarrow \mathscr{D}'(G)$  defined  by
 \begin{equation*}
   \partial_{Y}f(x)=  (Y_{x}f)(x)=\frac{d}{dt}f(x\exp(tY) )|_{t=0}.
 \end{equation*}If $\{X_1,\cdots, X_n\}$ is a basis of the Lie algebra $\mathfrak{g},$ we use the standard multi-index notation
 $$ \partial_{X}^{\alpha}=X_{x}^{\alpha}=\partial_{X_1}^{\alpha_1}\cdots \partial_{X_n}^{\alpha_n}.     $$

By using this property, together with the following notation for the so-called  elliptic weight $$\langle\xi \rangle:=(1+\lambda_{[\xi]})^{1/2},\,\,[\xi]\in \widehat{G},$$ we can finally give the definition of global symbol classes. Here, $\lambda_{[\xi]},$ $[\xi]\in \widehat{G},$ denotes the corresponding eigenvalue of the positive Laplacian  (in a bijective manner) indexed by an equivalence class $[\xi]\in \widehat{G}.$ 
\begin{definition}\label{def.elliptic.classes}
Let $0\leqslant \delta,\rho\leqslant 1.$ Let $$\sigma:G\times \widehat{G}\rightarrow \bigcup_{[\xi]\in \widehat{G}}\mathbb{C}^{d_\xi\times d_\xi},$$ be a matrix-valued function such that for any $[\xi]\in \widehat{G},$ $\sigma(\cdot,[\xi])$ is smooth, and such that, for any element $x\in G$ there is a distribution $k_{x}\in \mathscr{D}'(G),$ of $C^\infty$-class  in $x,$ satisfying that $\sigma(x,\xi)=\widehat{k}_{x}(\xi),$ $[\xi]\in \widehat{G}$.  The collection of all matrix-valued symbols $\sigma=\sigma(x,\xi)$ satisfying these properties will be denoted  by $\mathscr{D}'(G\times \widehat{G}).$

We say that $\sigma \in \mathscr{S}^{m}_{\rho,\delta}(G)$ if, for all $\beta$ and  $\gamma $ multi-indices and for all $(x,[\xi])\in G\times \widehat{G}$,  the following  inequalities 
\begin{equation}\label{HormanderSymbolMatrix}
   \Vert \partial_{X}^\beta \Delta_\xi^{\gamma} \sigma(x,\xi)\Vert_{\textnormal{op}}\leqslant C_{\alpha,\beta}
    \langle \xi \rangle^{m-\rho|\gamma|+\delta|\beta|},
\end{equation} hold,
 where $\|\cdot\|_{op}$ denotes the $\ell^2\rightarrow\ell^2$ operator norm 
    \begin{equation}\label{defn.opnorm}
        \|\sigma(x,\xi)\|_{op}=\sup \{ \|\sigma(x,\xi)v\|_{\ell^2}: v\in \mathbb{C}^{d_\xi},\|v\|_{\ell^2}=1\}.
    \end{equation}
For $\sigma_A\in \mathscr{S}^m_{\rho,\delta}(G)$ we will write $A\in\Psi^m_{\rho,\delta}(G)\equiv\textnormal{Op}(\mathscr{S}^m_{\rho,\delta}(G)).$
\end{definition}

The global H\"ormander classes on compact Lie groups  describe the H\"ormander classes defined by local coordinate systems. We present the corresponding statement as follows. 
\begin{theorem}[Equivalence of classes, \cite{Ruz,RuzhanskyTurunenWirth2014}] Let $A:C^{\infty}(G)\rightarrow\mathscr{D}'(G)$ be a continuous linear operator and let us consider  $0\leq \delta<\rho\leq 1,$ with $\rho\geq 1-\delta.$ Then, $A\in \Psi^m_{\rho,\delta}(G,\textnormal{loc}),$ if and only if $\sigma_A\in \mathscr{S}^m_{\rho,\delta}(G),$ consequently
\begin{equation}\label{EQequivalence}
   \textnormal{Op}(\mathscr{S}^m_{\rho,\delta}(G))= \Psi^m_{\rho,\delta}(G,\textnormal{loc}),\,\,\,0\leqslant \delta<\rho\leqslant 1,\,\rho\geqslant   1-\delta.
\end{equation}
\end{theorem}

\subsection{Subelliptic H\"ormander classes on compact Lie groups} 
 In order to define the subelliptic H\"ormander classes, we will use a suitable basis of the Lie algebra arising from Taylor expansions.  We explain the choice of this basis by means of the following lemma (see \cite[Section 3.1]{CR20}).

 \begin{lemma}\label{Taylorseries} Let $G$ be a compact Lie group of dimension $n.$ Let $\mathfrak{D}=\{\Delta_{q_{(j)}}\}_{1\leqslant j\leqslant n}$ be a strongly admissible collection of difference operators (for the definition see \eqref{def.admissible.diff} and \eqref{def.stongly.admissible.diff}). 
Then there exists a basis $X_{\mathfrak{D}}=\{X_{1,\mathfrak{D}},\cdots ,X_{n,\mathfrak{D}}\}$ of $\mathfrak{g}$ such that $$X_{j,\mathfrak{D}}q_{(k)}(\cdot^{-1})(e)=\delta_{jk}.
$$
Moreover, by using the multi-index notation $$\partial_{X}^{(\beta)}=\partial_{X_{1,\mathfrak{D}}}^{\beta_1}\cdots \partial_{X_{n,\mathfrak{D}}}^{\beta_n}, $$ for any $\beta\in\mathbb{N}_0^n,$
where $$\partial_{X_{i,\mathfrak{D}}}f(x)=  \frac{d}{dt}f(x\exp(tX_{i,\mathfrak{D}}) )|_{t=0},\,\,f\in C^{\infty}(G),$$ and denoting by
\begin{equation*}
    R_{x,N}^{f}(y)=f(xy)-\sum_{|\alpha|<N}q_{(1)}^{\alpha_1}(y^{-1})\cdots q_{(n)}^{\alpha_n}(y^{-1})\partial_{X}^{(\alpha)}f(x)
\end{equation*} 
the Taylor remainder, we have that 
\begin{equation*}
    | R_{x,N}^{f}(y)|\leqslant C|y|^{N}\max_{|\alpha|\leqslant N}\Vert \partial_{X}^{(\alpha)}f\Vert_{L^\infty(G)},
\end{equation*}
where the constant $C>0$ is dependent on $N,$ $G$ and $\mathfrak{D}$ (but not on $f\in C^\infty(G)).$ In addition we have that $\partial_{X}^{(\beta)}|_{x_1=x}R_{x_1,N}^{f}=R_{x,N}^{\partial_{X}^{(\beta)}f}$, and 
\begin{equation*}
    | \partial_{X}^{(\beta)}|_{y_1=y}R_{x,N}^{f}(y_1)|\leqslant C|y|^{N-|\beta|}\max_{|\alpha|\leqslant N-|\beta|}\Vert \partial_{X}^{(\alpha+\beta)}f\Vert_{L^\infty(G)},
\end{equation*}provided that $|\beta|\leqslant N.$
 \end{lemma}

Denoting by $\Delta_{\xi}^\alpha:=\Delta_{q_{(1)}}^{\alpha_1}\cdots   \Delta_{q_{(n)}}^{\alpha_{n}},$ we can introduce the subelliptic H\"ormander class of symbols of order $m\in \mathbb{R}$ and of type $(\rho,\delta)$. We will use the notation    $\widehat{ \mathcal{M}}$ to indicate   the matrix symbol of  $\mathcal{M}:=(1+\mathcal{L})^{\frac{1}{2}}.$ Also, for every $[\xi]\in \widehat{G}$ and $s\in \mathbb{R},$ we define the subellliptic matrix weight,
   \begin{equation*}
       \widehat{ \mathcal{M}}(\xi)^{s}:=\textnormal{diag}[(1+\nu_{ii}(\xi)^2)^{\frac{s}{2}}]_{1\leqslant i\leqslant d_\xi},
   \end{equation*} 
   where $\widehat{\mathcal{L}}(\xi)=:\textnormal{diag}[\nu_{ii}(\xi)^2]_{1\leqslant i\leqslant d_\xi}$ is the symbol of the sub-Laplacian $\mathcal{L}$ at $[\xi]$, as the symbol of the operator $\mathcal{M}_s:=(1+\mathcal{L})^{\frac s 2}.$

\begin{definition}[Subelliptic H\"ormander classes]\label{contracted''}
   Let $G$ be a compact Lie group and let $0\leqslant \delta,\rho\leqslant 1.$ Let us consider a sub-Laplacian $\mathcal{L}=-(X_1^2+\cdots +X_{k}^2)$ on $G,$ where the system of vector fields $X=\{X_i\}_{i=1}^{k}$ satisfies the H\"ormander condition of step $\kappa$.  We say that $\sigma \in {S}^{m,\,\mathcal{L}}_{\rho,\delta}(G\times \widehat{G})$ if for all $\alpha, \beta\in \mathbb{N}^n_0$ 
   \begin{equation}\label{InIC}
      p_{\alpha,\beta,\rho,\delta,m,\textnormal{left}}(\sigma)':= \sup_{(x,[\xi])\in G\times \widehat{G} }\Vert \widehat{ \mathcal{M}}(\xi)^{(\rho|\alpha|-\delta|\beta|-m)}\partial_{X}^{(\beta)} \Delta_{\xi}^{\alpha}\sigma(x,\xi)\Vert_{\textnormal{op}} <\infty,
   \end{equation}
   \begin{equation}\label{InIIC}
      p_{\alpha,\beta,\rho,\delta,m,\textnormal{right}}(\sigma)':= \sup_{(x,[\xi])\in G\times \widehat{G} }\Vert (\partial_{X}^{(\beta)} \Delta_{\xi}^{\alpha} \sigma(x,\xi) ) \widehat{ \mathcal{M}}(\xi)^{(\rho|\alpha|-\delta|\beta|-m)}\Vert_{\textnormal{op}} <\infty,
   \end{equation} 
   where $\|\cdot\|_{\mathrm{op}}$ is as in \eqref{defn.opnorm}.
  \end{definition}
The previous Definition \ref{contracted''} can be simplified in the sense that \eqref{InIC} and \eqref{InIIC} are equivalent conditions. Moreover,  The following conditions are equivalent, see Theorem 4.4 of \cite{CR20}.
\begin{itemize}
    \item[A.] For every $\alpha,\beta\in \mathbb{N}_0^n,$ \begin{equation}\label{InI2}
      p_{\alpha,\beta,\rho,\delta,m,\textnormal{left}}(a):= \sup_{(x,[\xi])\in G\times \widehat{G} }\Vert \widehat{ \mathcal{M}}(\xi)^{\frac{1}{\kappa}(\rho|\alpha|-\delta|\beta|-m)}\partial_{X}^{(\beta)} \Delta_{\xi}^{\alpha}a(x,\xi)\Vert_{\textnormal{op}} <\infty.
   \end{equation}
   \item[B.]For every $\alpha,\beta\in \mathbb{N}_0^n,$ \begin{equation}\label{InII2}
      p_{\alpha,\beta,\rho,\delta,m,\textnormal{right}}(a):= \sup_{(x,[\xi])\in G\times \widehat{G} }\Vert (\partial_{X}^{(\beta)} \Delta_{\xi}^{\alpha} a(x,\xi) ) \widehat{ \mathcal{M}}(\xi)^{\frac{1}{\kappa}(\rho|\alpha|-\delta|\beta|-m)}\Vert_{\textnormal{op}} <\infty.
   \end{equation}
   \item[C.] For all $r\in \mathbb{R},$ $\alpha,\beta\in \mathbb{N}_0^n,$
    \begin{equation}\label{InI2X}
      p_{\alpha,\beta,\rho,\delta,m,r}(a):= \sup_{(x,[\xi])\in G\times \widehat{G} }\Vert \widehat{ \mathcal{M}}(\xi)^{\frac{1}{\kappa}(\rho|\alpha|-\delta|\beta|-m-r)}\partial_{X}^{(\beta)} \Delta_{\xi}^{\alpha}a(x,\xi)\widehat{ \mathcal{M}}(\xi)^{\frac{r}{\kappa}}\Vert_{\textnormal{op}} <\infty.
   \end{equation}
   \item[D.] There exists $r_0\in \mathbb{R},$ such that for every $\alpha,\beta\in \mathbb{N}_0^n,$
    \begin{equation}\label{InI2X''}
      p_{\alpha,\beta,\rho,\delta,m,r_0}(a):= \sup_{(x,[\xi])\in G\times \widehat{G} }\Vert \widehat{ \mathcal{M}}(\xi)^{\frac{1}{\kappa}(\rho|\alpha|-\delta|\beta|-m-r_0)}\partial_{X}^{(\beta)} \Delta_{\xi}^{\alpha}a(x,\xi)\widehat{ \mathcal{M}}(\xi)^{\frac{r_0}{\kappa}}\Vert_{\textnormal{op}} <\infty.
   \end{equation}
\end{itemize}

By following the usual nomenclature, we  define:
\begin{equation*}
    \textnormal{Op}({S}^{m,\,\mathcal{L}}_{\rho,\delta}(G\times \widehat{G})):=\{A:C^{\infty}(G)\rightarrow \mathscr{D}'(G):\sigma_A\equiv\widehat{A}(x,\xi)\in {S}^{m,\,\mathcal{L}}_{\rho,\delta}(G\times \widehat{G}) \},
\end{equation*} with
\begin{equation*}
    Af(x)=\sum_{[\xi]\in \widehat{G}}d_\xi \textnormal{{Tr}}(\xi(x)\widehat{A}(x,\xi)\widehat{f}(\xi)),\,\,\,f\in C^\infty(G),\,x\in G.  
\end{equation*}

The decay properties of subelliptic symbols are summarized in the following lemma (see \cite[Chapter 4]{CR20}), where we present a necessary (but not a sufficient) condition in order that the matrix-symbol $a:=a(x,\xi)$ belongs to the class ${S}^{m,\,\mathcal{L}}_{\rho,\delta}(G\times \widehat{G}).$

\begin{lemma}\label{lemadecaying1}
Let $G$ be a compact Lie group and  let $0\leqslant \delta,\rho\leqslant 1.$ If $a\in {S}^{m,\,\mathcal{L}}_{\rho,\delta}(G\times \widehat{G}),$ then for every $\alpha,\beta\in \mathbb{N}_0^n,$ there exists $C_{\alpha,\beta}>0$ satisfying the estimates
\begin{equation*}
    \Vert \partial_{X}^{(\beta)} \Delta_{\xi}^{\alpha}a(x,\xi)\Vert_{\textnormal{op}}\leqslant C_{\alpha,\beta}\sup_{1\leqslant i\leqslant d_\xi}(1+\nu_{ii}(\xi)^2)^{\frac{m-\rho|\alpha|+\delta|\beta|}{2 }},
\end{equation*}uniformly in $(x,[\xi])\in G\times \widehat{G}.$ 
\end{lemma}

In the next theorem we describe the fundamental properties of the subelliptic calculus  \cite{CR20}, like compositions, adjoints, and boundedness properties.
\begin{theorem}\label{calculus} Let $0\leqslant \delta<\rho\leqslant 1,$ and let  $\Psi^{m,\,\mathcal{L}}_{\rho,\delta}:=\textnormal{Op}({S}^{m,\,\mathcal{L}}_{\rho,\delta}(G\times \widehat{G})),$ for every $m\in \mathbb{R}.$ Then,
\begin{itemize}
    \item [-] The mapping $A\mapsto A^{*}:\Psi^{m,\,\mathcal{L}}_{\rho,\delta}\rightarrow \Psi^{m,\,\mathcal{L}}_{\rho,\delta}$ is a continuous linear mapping between Fr\'echet spaces and  the  symbol of $A^*,$ $\sigma_{A^*}(x,\xi)$ satisfies the asymptotic expansion,
 \begin{equation*}
    \widehat{A^{*}}(x,\xi)\sim \sum_{|\alpha|= 0}^\infty\Delta_{\xi}^\alpha\partial_{X}^{(\alpha)} (\widehat{A}(x,\xi)^{*}).
 \end{equation*} This means that, for every $N\in \mathbb{N},$ and for all $\ell\in \mathbb{N},$
\begin{equation*}
   \Small{ \Delta_{\xi}^{\alpha_\ell}\partial_{X}^{(\beta)}\left(\widehat{A^{*}}(x,\xi)-\sum_{|\alpha|\leqslant N}\Delta_{\xi}^\alpha\partial_{X}^{(\alpha)} (\widehat{A}(x,\xi)^{*}) \right)\in {S}^{m-(\rho-\delta)(N+1)-\rho\ell+\delta|\beta|,\,\mathcal{L}}_{\rho,\delta}(G\times\widehat{G}) },
\end{equation*} where $|\alpha_\ell|=\ell.$
\item [-] The mapping $(A_1,A_2)\mapsto A_1\circ A_2: \Psi^{m_1,\,\mathcal{L}}_{\rho,\delta}\times \Psi^{m_2,\,\mathcal{L}}_{\rho,\delta}\rightarrow \Psi^{m_1+m_2,\,\mathcal{L}}_{\rho,\delta}$ is a continuous bilinear mapping between Fr\'echet spaces, and the symbol of $A=A_{1}\circ A_2$  is given by the asymptotic formula
\begin{equation*}
    \sigma_A(x,\xi)\sim \sum_{|\alpha|= 0}^\infty(\Delta_{\xi}^\alpha\widehat{A}_{1}(x,\xi))(\partial_{X}^{(\alpha)} \widehat{A}_2(x,\xi)),
\end{equation*}which, in particular, means that, for every $N\in \mathbb{N},$ and for all $\ell \in\mathbb{N},$
\begin{align*}
    &\Delta_{\xi}^{\alpha_\ell}\partial_{X}^{(\beta)}\left(\sigma_A(x,\xi)-\sum_{|\alpha|\leqslant N}  (\Delta_{\xi}^\alpha\widehat{A}_{1}(x,\xi))(\partial_{X}^{(\alpha)} \widehat{A}_2(x,\xi))  \right)\\
    &\hspace{2cm}\in {S}^{m_1+m_2-(\rho-\delta)(N+1)-\rho\ell+\delta|\beta|,\,\mathcal{L}}_{\rho,\delta}(G\times \widehat{G}),
\end{align*}for all  $\alpha_\ell \in \mathbb{N}_0^n$ with $|\alpha_\ell|=\ell.$
\item [-] For  $0\leqslant \delta< \rho\leqslant    1,$  (or for $0\leq \delta\leq \rho\leq 1,$  $\delta<1/\kappa$) let us consider a continuous linear operator $A:C^\infty(G)\rightarrow\mathscr{D}'(G)$ with symbol  $\sigma\in {S}^{0,\,\mathcal{L}}_{\rho,\delta}(G\times \widehat{G})$. Then $A$ extends to a bounded operator from $L^2(G)$ to  $L^2(G).$ 
\end{itemize}
\end{theorem}

Finally, we present the following result about the $L^p$-boundedness of the subelliptic classes for the sub-Laplacian $\mathcal{L},$ see \cite[Section 6]{CR20}.
 \begin{theorem}\label{Lp1CardonaDelgadoRuzhansky2}
Let $G$ be a compact Lie group and let us denote by $Q$ the Hausdorff
dimension of $G$ associated to the control distance associated to the sub-Laplacian $\mathcal{L}=\mathcal{L}_X,$ where  $X= \{X_{1},\cdots,X_k\} $ is a system of vector fields satisfying the H\"ormander condition of order $\kappa$.  For  $0\leqslant \delta<\rho\leqslant 1,$  let us consider a continuous linear operator $A:C^\infty(G)\rightarrow\mathscr{D}'(G)$ with symbol  $\sigma\in {S}^{-m,\mathcal{L}}_{\rho,\delta}(G\times \widehat{G})$, $m\geqslant   0$. Then $A$ extends to a bounded operator on $L^p(G)$ provided that 
$$
    m\geqslant   m_p:= Q(1-\rho)\left|\frac{1}{p}-\frac{1}{2}\right|.
$$
\end{theorem}

\subsection{$L^p$-$L^q$-boundedness for Bessel potentials} Here we discuss  the sharpness of the Hardy-Littlewood-Sobolev inequality, in this case, formulated in terms of the $L^p$-$L^q$-boundedness of Bessel potentials. 

\begin{lemma}\label {Lemma:Bessel:potential:lplq}Let $1<p,q<\infty.$ Let $G$ be a compact Lie group of the Hausdorff dimension $Q$ associated to the control distance associated to the sub-Laplacian $\mathcal{L}=-\sum_{1\leq i\leq k}X_i^2.$ Then, the Bessel operator $B_{a}=(1+\mathcal{L})^{-\frac{a}{2}},$ admits a bounded extension from $L^p(G)$ into $L^q(G),$ that is, the estimate
\begin{equation}\label{eqref:lplq}
    \Vert B_{a} f\Vert_{L^q}\leq  C\Vert  f\Vert_{L^p}\,\,\,
\end{equation} holds, if and only if, $1<p<q<\infty$ and
\begin{equation}\label{Necessary:condition:2}
   a\geq Q\left(\frac{1}{p}- \frac{1}{q}\right).
\end{equation}    
\end{lemma}
\begin{proof} The sufficiency of the condition \eqref{Necessary:condition:2} on $a$ for the $L^p$-$L^q$-boundedness of $B_a$ is exactly the Hardy-Littlewood-Sobolev inequality, see e.g. \cite{Saloff-CosteBook}. On the other hand, assume that $$B_a:L^p(G)\rightarrow L^q(G)$$ is bounded. 

Using the subelliptic functional calculus in \cite[Section 8]{CR20}, we have that $  \sqrt{\mathcal{L}}\in \Psi^{1,\mathcal{L}}_{1,0}(G\times \widehat{G}) $  is a pseudo-differential operator of first order. 
Since $\sqrt{\mathcal{L}}$ is not invertible, let $\sqrt{\mathcal{L}}^{-1}$ be the inverse of $\sqrt{\mathcal{L}}$ on the orthogonal complement of its kernel, i.e. if $P_0$ is the $L^2$-orthogonal projection on $\textnormal{Ker}(\sqrt{\mathcal{L}}),$ then
$$ \sqrt{\mathcal{L}}^{-1}\sqrt{\mathcal{L}}=\sqrt{\mathcal{L}}\sqrt{\mathcal{L}}^{-1}=I-P_0,\,\,\forall f\in \textnormal{Ker}(\sqrt{\mathcal{L}}),\, \sqrt{\mathcal{L}}^{-1}f:=0.  $$
This operator agrees with the operator $f(\mathcal{L}),$ defined by the spectral calculus where $f(t)=t^{-1},$ and that $\sqrt{\mathcal{L}}^{-1}\in \Psi^{-1,\mathcal{L}}_{1,0}(G\times \widehat{G}) $  is a consequence of the functional calculus in \cite[Section 8]{CR20}. 

Let us consider the operator $\mathcal{L}^{-\frac{a}{2}}:=(\sqrt{\mathcal{L}}_{G}^{-1})^a\in \Psi^{-a,\mathcal{L}}_{1,0}(G\times \widehat{G}) $ defined by the spectral calculus if $a>0$. In the case where $a<0,$ $\mathcal{L}^{-\frac{a}{2}}\in \Psi^{|a|,\mathcal{L}}_{1,0}(G\times \widehat{G}) .$ In view of the inclusion of the powers $\mathcal{L}^{-a}$ to the subelliptic calculus, note that $$\mathcal{L}^{-\frac{a}{2}}=\mathcal{L}^{-\frac{a}{2}}B_{-a}B_{a}=TB_a, $$ where
$$ T=\mathcal{L}^{-\frac{a}{2}}B_{-a}=\mathcal{L}^{-\frac{a}{2}} (1+\mathcal{L})^{\frac{a}{2}} \in \Psi^{0}_{1,0}(G\times \widehat{G}) $$
is a subelliptic pseudo-differential operator of order zero (see \cite[Theorem 8.20]{CR20}). Then $T:L^q(G)\rightarrow L^q(G)$ is bounded  and from the $L^p$-$L^q$ boundedness of $B_a$ we deduce the $L^p$-$L^q$-boundedness of $\mathcal{L}^{-\frac{a}{2}}=TB_a.$ Note that we have the validity of the estimate
\begin{equation}
    \Vert \mathcal{L}^{-\frac{a}{2}} f\Vert_{L^q}\leq  C\Vert  f\Vert_{L^p}.
\end{equation}In other words, we have that the inequality 
\begin{equation} \label{eq2.21}
    \Vert  f\Vert_{L^q}\leq  C\Vert \mathcal{L}^{\frac{a}{2}}  f\Vert_{L^p},
\end{equation} is valid with $C>0,$ independent of $f\in L^q(G).$ Define the semigroup
$$T_t=e^{-t\mathcal{L}}, t>0,$$ and consider the heat kernel $h_t$ defined by $T_tf=f\ast p_t.$ Note that (see \cite[Lemma VIII.2.5, Page 110]{Saloff-CosteBook}) 
$$\sup_{t>0}\Vert h_t\Vert_{L^1}\lesssim 1.$$ In consequence
we have that
\begin{equation}\label{Equi:l1}
  \sup_{t>0}  \Vert T_t f\Vert_{L^1(G)}=\sup_{t>0}\Vert f\ast h_t\Vert_{L^1(G)}\leq \sup_{t>0} \Vert h_t\Vert_{L^1(G)}\Vert f\Vert_{L^1(G)}\lesssim \Vert f\Vert_{L^1(G)}.
\end{equation}In a similar way
\begin{equation}\label{Equi:inf}
  \sup_{t>0}  \Vert T_t f\Vert_{L^\infty(G)}=\sup_{t>0}\Vert f\ast h_t\Vert_{L^\infty(G)}\leq \sup_{t>0} \Vert h_t\Vert_{L^1(G)}\Vert f\Vert_{L^\infty(G)}\lesssim \Vert f\Vert_{L^\infty(G)}.
\end{equation}In view of \eqref{Equi:l1} and \eqref{Equi:inf} we have that the semigroup $T_t$ is equicontinuous on $L^1$ and on $L^\infty.$ Moreover, in view of Remark \ref{Remark}, with $X=\{X_1,X_2,\cdots,X_k\},$ the semigroup $T_{t}=e^{\Delta_X}=e^{-t\mathcal{L}}$ is a submarkovian semigroup. Then, in view of \eqref{eq2.21}, by applying Theorem \ref{Converse:HLT}, with $\tilde{Q}$ defined by
\begin{equation}\label{a:tilde:Q}
    a=\tilde{Q}\left(\frac{1}{p}-\frac{1}{q}\right),
\end{equation}we have the estimate
\begin{equation}\label{semigroup}
    \Vert e^{-t\mathcal{L}} \Vert_{L^1\rightarrow L^\infty}\leq C_{\tilde{Q}}t^{-\tilde{Q}/2}.
\end{equation}Note that if
$$Q':=\inf\{\tilde{Q}: T_t \textnormal{ satisfies } \eqref{semigroup}\textnormal{ for all }t: 0<t<1 \},$$ 
then $Q'\leq \tilde{Q},$
$$  \Vert e^{-t\mathcal{L}} \Vert_{L^1\rightarrow L^\infty}\leq  C_{{Q}'}t^{-{Q}'/2}\leq   C_{\tilde{Q}}t^{-\tilde{Q}/2}, \,0<t<1. $$ However, by the sharpness of the heat kernel estimates is very well known that the infimum $Q'$ agrees with the Hausdorff dimension $Q$ of the group associated to the control distance associated to the sub-Laplacian $\mathcal{L}$ (see \cite[Chapter VIII]{Saloff-CosteBook}), that is $Q'=Q.$ Since $\tilde{Q}\geq Q,$ in view of \eqref{a:tilde:Q} we have proved that
\begin{equation}
    a\geq Q\left(\frac{1}{p}-\frac{1}{q}\right),
\end{equation} as desired.
\end{proof}

\section{$L^p$-$L^q$-boundedness of pseudo-differential operators}\label{Section:lp:lq}

\subsection{$L^p$-$L^q$-boundedness of pseudo-differential operators I}
The following result presents the necessary and sufficient criteria for a pseudo-differential operator to be bounded from $L^p(G)$ into $L^q(G)$ for the range $1<p\leq 2 \leq q<\infty.$

\begin{theorem} \label{thh1}
    
Let $1<p\leq 2\leq q<\infty$ and $m \in \mathbb{R}.$  Let $G$ be a compact Lie group, and let $Q$ be its Hausdorff dimension with respect to the control distance associated to a H\"ormander sub-Laplacian $\mathcal{L}.$ Let $0\leq \delta< \rho\leq 1.$  Then, every pseudo-differential operator $A\in \Psi^{m,\mathcal{L}}_{\rho,\delta}(G\times \widehat{G})$ with $0\leq \delta <  \rho \leq 1$  admits a bounded extension from $L^p(G)$ into $L^q(G),$ that is
\begin{equation}
   \forall f\in C^{\infty}_0(G),\, \Vert A f\Vert_{L^q}\leq  C\Vert  f\Vert_{L^p}\,\,\,
\end{equation} holds, if and only if, 
\begin{equation}\label{Necessary:condition:3a}
   m\leq -Q\left(\frac{1}{p}- \frac{1}{q}\right).
\end{equation}    
\end{theorem}
\begin{proof}
    Assume that $m> -Q\left(\frac{1}{p}- \frac{1}{q}\right).$ We are going to show that there exists $A\in \Psi^{m, \mathcal{L}}_{\rho,\delta}(G\times \widehat{G})$ which is not bounded from $L^p(G)$ into $L^q(G).$ We consider $$A=B_{-m}=(1+\mathcal{L})^{\frac{m}{2}}\in \Psi^{m,\mathcal{L}}_{1,0}(G\times \widehat{G})\subset \Psi^{m,\mathcal{L}}_{\rho,\delta}(G\times \widehat{G}).$$ Since $$-m< Q\left(\frac{1}{p}- \frac{1}{q}\right),$$ from Lemma \ref{Lemma:Bessel:potential:lplq}, we have that  $A=B_{-m}$ is not bounded from $L^p(G)$ into $L^q(G).$ So, we have proved the necessity of the order condition  \eqref{Necessary:condition:3a}. Now, in order to prove the reverse statement, we consider $m$ satisfying \eqref{Necessary:condition:3a} and  $m_1$ and $m_2$ satisfying the conditions
    \begin{equation}
     m=m_1+m_2,\,   m_1\leq -Q(1/p-1/2),\,m_2\leq -Q(1/2-1/q).
    \end{equation} If  $A\in \Psi^{m,\mathcal{L}}_{\rho,\delta}(G\times \widehat{G}),$ we factorise $A$ as follows,
    \begin{align*}
        A=B_{-m_2}A_0 B_{-m_1},\,\,A_0=B_{m_2}AB_{m_1}.
    \end{align*}Note that $A_0\in \Psi^{0,\mathcal{L}}_{\rho,\delta}(G\times \widehat{G}).$ The Calder\'on-Vaillancourt theorem (Theorem \ref{calculus}(iii)) implies that $A_0$ is bounded from $L^2(G)$ into $L^2(G).$ On the other hand, from Lemma \ref{Lemma:Bessel:potential:lplq} we have that $B_{m_2}:L^2(G)\rightarrow L^q(G),$ and $B_{m_1}:L^p(G)\rightarrow L^2(G),$ are bounded operators. In consequence, we have proved that $A$ admits a bounded extension from $L^p(G)$ into $L^q(G).$ The proof is complete.    
\end{proof}

\subsection{$L^p$-$L^q$-boundedness of pseudo-differential operators II}
In this subsection, we consider the $L^p-L^q$ boundedness of pseudo-differential operators on compact Lie groups for a wider range of indices $p$ and $q.$ Our main result of this section is the following theorem. 
\begin{theorem} \label{thh2} Let $1<p \leq q <\infty, m \in \mathbb{R},$ and let $0\leq \delta< \rho\leq 1.$  Let $G$ be a compact Lie group, and let $Q$ be its Hausdorff dimension with respect to the control distance associated to a H\"ormander sub-Laplacian $\mathcal{L}.$   Then, every pseudo-differential operator $A\in \Psi^{m,\mathcal{L}}_{\rho,\delta}(G\times \widehat{G})$ admits a bounded extension from $L^p(G)$ into $L^q(G),$ that is
\begin{equation}
   \forall f\in C^{\infty}_0(G),\, \Vert A f\Vert_{L^q}\leq  C\Vert  f\Vert_{L^p}\,\,\,
\end{equation} holds in the following cases: 
\begin{itemize}
        \item[(i)] if $1<p\leq q \leq 2$ and  \begin{equation}\label{Necessary:condition:4}
   m\leq -Q \left( \frac{1}{p}-\frac{1}{q}+(1-\rho) \left(\frac{1}{q}-\frac{1}{2}\right)\right). \end{equation}    
\item[(ii)] if $2 \leq p \leq q<\infty$ and 

\begin{equation}\label{Necessary:condition:3}
   m\leq -Q\left( \frac{1}{p}-\frac{1}{q}+(1-\rho) \left(\frac{1}{2}-\frac{1}{p}\right)\right);
   \end{equation} 
\end{itemize}
\end{theorem}
\begin{proof}  
\begin{itemize}
    \item[(i)] Let us consider $p, q$ and $m$ satisfying the conditions given in (i). Choose $m'= -Q \left( \frac{1}{p}-\frac{1}{q} \right)$ and this implies that the subelliptic Bessel potential $B_{-m'}$ is bounded from $L^p(G)$ to $L^q(G)$ as a consequence of Lemma \ref{Lemma:Bessel:potential:lplq}. For $A \in \Psi^{m, \mathcal{L}}_{\rho, \delta}(G \times \widehat{G}),$ we decompose it as follows:
    $$A= (A B_{m'}) B_{-m'}.$$
    Now, we note that operator $AB_{m'} \in \Psi^{m-m',\mathcal{L}}_{\rho, \delta}(G \times \widehat{G})$ with $m-m'$ satisfying $m-m' \leq -Q(1-\rho) \left( \frac{1}{q}-\frac{1}{2} \right).$ Then, Theorem \ref{Lp1CardonaDelgadoRuzhansky2} shows that $AB_{m'}$ is bounded operator from $L^q(G)$ into $L^q(G).$ Therefore, we conclude that the operator $A$ has a bounded extension from $L^p(G)$ into $L^q(G).$
    \item[(ii)] To prove this part we follow the same strategy as in Part (i). We factorise the operator $A \in \Psi^{m,\mathcal{L}}_{\rho, \delta}(G \times \widehat{G})$  in the following manner:
    $$A= B_{-m'} (B_{m'} A),$$
    where $m'=-Q(\frac{1}{p}-\frac{1}{q}).$ Again, it follows from Lemma \ref{Lemma:Bessel:potential:lplq} that the operator $B_{-m'}$ is a bounded from $L^{p}(G)$ into $L^q(G).$ On the other hand, the operator $B_{m'} A \in \Psi^{m-m',\mathcal{L}}_{\rho, \delta}(G \times \widehat{G})$ with $m-m'\leq -Q(1-\rho) \left(\frac{1}{2}-\frac{1}{p} \right),$ which, as a consequence of Theorem \ref{Lp1CardonaDelgadoRuzhansky2}, yields that the operator $B_{-m'} A$ is bounded from $L^p(G)$ into $L^p(G).$ Hence, we conclude that the operator $A$ has a bounded extension from $L^p(G)$ into $L^q(G).$ 
\end{itemize}This completes the proof of this theorem.  \end{proof}

\section{$L^p$-$L^q$-boundedness of pseudo-differential operators on $\mathbb{S}^3$ and on $\textnormal{SU}(3)$}\label{ex}

We will present an explicit form of our main Theorem \ref{main:theorem} on the sphere $\textnormal{SU}(2)\cong \mathbb{S}^3$ and on $\textnormal{SU}(3).$ By abuse of notation, we will use the same symbol to denote an element of the Lie algebra and the vector field on the group  obtained by left translation.

\subsection{$L^p$-$L^q$-boundedness of pseudo-differential operators on $\mathbb{S}^3$}
Let us consider the left-invariant first-order  differential operators $\partial_{+},\partial_{-},\partial_{0}: C^{\infty}(\textnormal{SU}(2))\rightarrow C^{\infty}(\textnormal{SU}(2)),$ called creation, annihilation, and neutral operators respectively, (see Definition 11.5.10 of \cite{Ruz}) and let us define 
\[ 
    X_{1}=-\frac{i}{2}(\partial_{-}+\partial_{+}),\, X_{2}=\frac{1}{2}(\partial_{-}-\partial_{+}),\, X_{3}=-i\partial_0,
\]where $X_{3}=[X_1,X_2],$ based on the commutation relations $[\partial_{0},\partial_{+}]=\partial_{+},$ $[\partial_{-},\partial_{0}]=\partial_{-},$ and $[\partial_{+},\partial_{-}]=2\partial_{0}.$ The system $X=\{X_1,X_2\}$ satisfies the H\"ormander condition of step $\kappa=2,$ and the Hausdorff dimension defined by the control distance associated to the sub-Laplacian   $\mathcal{L}_1=-X_1^2-X_2^2$  is $Q=4.$ In a similar way, we can define the sub-Laplacian $\mathcal{L}_2=-X_2^2-X_3^2$ associated to the system of vector fields $X'=\{X_2,X_3\},$ which also satisfies the H\"ormander condition of step $\kappa=2.$ In the following Corollary we describe the $L^p$-$L^q$-boundedness of subelliptic pseudo-differential operators on $\textnormal{SU}(2)\cong \mathbb{S}^3.$ In this case we observe that one can identify $\widehat{\textnormal{SU}}(2)\cong \frac{1}{2}\mathbb{N}_0, $ see \cite{Ruz} for details.

\begin{corollary}\label{main:theorem:su:2} 
Let $1<p, q<\infty,$ and   $0\leq \delta< \rho\leq 1.$ Let us consider the H\"ormander sub-Laplacian $\mathcal{L}=-X_1^2-X_2^2.$  Then, the following statements hold.
\begin{itemize}
    \item Let $1<p\leq2\leq  q<\infty.$ Every pseudo-differential operator $$A\in \Psi^{m,\mathcal{L}}_{\rho,\delta}(\textnormal{SU}(2)\times \frac{1}{2}\mathbb{N}_0)$$ admits a bounded extension from $L^p(\textnormal{SU}(2))$ into $L^q(\textnormal{SU}(2)),$ that is
\begin{equation}\label{Lp-Lq:bound:su:2}
   \forall f\in C^{\infty}(\textnormal{SU}(2)),\, \Vert A f\Vert_{L^q(\textnormal{SU}(2))}\leq  C\Vert  f\Vert_{L^p(\textnormal{SU}(2))}\,\,\,
\end{equation} holds, if and only if, 
\begin{equation}\label{Necessary:condition:3:intro:su:2}
   m\leq -4\left(\frac{1}{p}- \frac{1}{q}\right).
\end{equation}  
\item Every pseudo-differential operator $A\in \Psi^{m,\mathcal{L}}_{\rho,\delta}(\textnormal{SU}(2)\times \frac{1}{2}\mathbb{N}_0)$ admits a bounded extension from $L^p(\textnormal{SU}(2))$ into $L^q(\textnormal{SU}(2)),$ that is \eqref{Lp-Lq:bound:su:2} holds, in the following cases: 
\begin{itemize}
        \item[(i)] if $1<p\leq q \leq 2$ and  \begin{equation}\label{Necessary:condition:4:intro:su:2}
   m\leq -4 \left( \frac{1}{p}-\frac{1}{q}+(1-\rho) \left(\frac{1}{q}-\frac{1}{2}\right)\right). \end{equation}    
\item[(ii)] if $2 \leq p \leq q<\infty$ and 

\begin{equation}\label{Necessary:condition:3:intro:su:2:2}
   m\leq -4\left( \frac{1}{p}-\frac{1}{q}+(1-\rho) \left(\frac{1}{2}-\frac{1}{p}\right)\right).
   \end{equation}
\end{itemize}
\end{itemize}
\end{corollary}
\subsection{$L^p$-$L^q$-boundedness of pseudo-differential operators on $\textnormal{SU}(3)$}
The special unitary group of $3\times 3$ complex matrices is defined by $$\textnormal{SU}(3)=\{g\in \textnormal{GL}(3,\mathbb{C}):gg^*=I_3\equiv(\delta_{ij})_{1\leqslant i,j\leqslant 3},\, \textnormal{\bf{det}}(g)=1\},$$
and its Lie algebra is given by
$$  \mathfrak{su}(3)=\{g\in \textnormal{GL}(3,\mathbb{C}):g+g^*=0,\, \textnormal{\bf{Tr}}(g)=0\}.$$ The inner product is defined by a multiple of the Killing form on $ \mathfrak{su}(3)$  given by $B(X,Y)=-\frac{1}{2}\textnormal{\bf{Tr}}[XY].$ The torus
$$ \mathbb{T}_{\textnormal{SU}(3)}=\{\textnormal{diag}[e^{i\theta_{1}},e^{i\theta_{2}},e^{i\theta_{3}}]:\theta_1+\theta_2+\theta_3=0,\,\theta_i\in \mathbb{R}\}  $$ is a maximal torus of $\textnormal{SU}(3),$ and its Lie algebra is given by
$$  \mathfrak{t}_{ \mathfrak{su}(3)}=\{\textnormal{diag}[i\theta_{1},i\theta_{2},i\theta_{3}]:\theta_1+\theta_2+\theta_3=0,\,\theta_i\in \mathbb{R}\}.  $$  The following vectors  
$$ T_{1}= \textnormal{diag}[-i,i,0],\quad T_{2}= \textnormal{diag}[-i/\sqrt{3},-i/\sqrt{3},2i/\sqrt{3}] $$ provide a basis for $ \mathfrak{t}_{ \mathfrak{su}(3)}.$ Completing this basis with the following vectors
\begin{align*}
  &X_{1}=  \begin{pmatrix}
0 & 1 &   0 \\
-1 & 0&  0 \\
0 & 0 & 0 
\end{pmatrix},\,
X_{2}=  \begin{pmatrix}
0 & i &   0 \\
i & 0&  0 \\
0 & 0 & 0 
\end{pmatrix},\,
\\
&X_{3}=  \begin{pmatrix}
0 & 0 &   0 \\
0 & 0&  1 \\
0 & -1 & 0 
\end{pmatrix},\,  X_{4}=  \begin{pmatrix}
0 & 0 &   0 \\
0 & 0&  -i \\
0 & -i & 0 
\end{pmatrix},\,
\\
    &X_{5}=  \begin{pmatrix}
0 & 0 &   1 \\
0 & 0&  0 \\
-1 & 0 & 0 
\end{pmatrix},\,
X_{6}=  \begin{pmatrix}
0 & 0 &   i \\
0 & 0&  0 \\
i & 0& 0 
\end{pmatrix},
\end{align*}
we obtain the  Gell-Mann system, which forms an orthonormal basis of $\mathfrak{su}(3).$ The system of vector fields $X=\{X_1,X_2,X_3,X_4,X_5,X_6\}$ satisfies the H\"ormander condition of step $\kappa=2,$ (see \cite[Section 11]{CR20}). Indeed, it can be deduced if we write
\[ 
  X_{7}=-[X_1,X_2]=  \begin{pmatrix}
-2i & 0 &   0 \\
0 & 2i&  0 \\
0 & 0 & 0 
\end{pmatrix},\,
\]
\[ 
    X_{8}=-[X_3,X_4]=  \begin{pmatrix}
0 & 0 &   0 \\
0 & 2i&  0 \\
0 & 0& -2i
\end{pmatrix}
\]
from T{\SMALL{ABLE}} \ref{table2}.
\begin{table}[h]  
\caption{Commutators in $\textnormal{SU}(3)$ } 
\centering 
\begin{tabular}{l c c rrrrrrr} 
\hline\hline   
       & \,\, $X_1$ & \,\, $X_2$  & \,\, $X_3$ & \,\, $X_4$ & \,\, $X_5$ & \,\, $X_6$ & \,\, $X_7$ &\,\, $X_8$    \\
$X_1$  &  0 & $-X_7$  & $X_5$   & $-X_6$ & $-X_3$  & $X_4$ & $4X_2$ & $2X_2$   \\

$X_2$  & $X_7$  &  0 & $X_6$ & $X_5$ & $-X_4$  & $-X_3$  & $-4X_1$ & $-2X_1$   \\

$X_3$  &  $-X_5$ &  $-X_6$ & 0 &  $-X_8$&  $X_1$& $X_2$ & $2X_4$  & $4X_4$   \\

$X_4$  & $X_6$ & $-X_5$  & $X_8$ & $0$ & $X_2$ & $-X_1$ &  $-2X_3$ & $-4X_3$   \\

$X_5$  & $X_3$ &  $X_4$ & $-X_1$ & $-X_2$ & 0  & $X_8-X_7$  & $2X_6$ & $-2X_6$   \\

$X_6$  & $-X_4$ & $X_3$  & $-X_2$  & $X_1$ &  $X_7-X_8$ & 0 & $-2X_5$ & $2X_5$    \\

$X_7$  & $-4X_2$ & $4X_1$  & $-2X_4$  & $2X_3$ & $-2X_6$ & $2X_5$ & 0 &  0 \\

$X_8$  & $2X_2$ &  $2X_1$ & $-4X_4$ & $4X_3$ & $2X_6 $  & $-2X_5$  & 0 & 0   \\[1ex]  
\hline 
\end{tabular}  
\label{table2}
\end{table}Observe that the Hausdorff dimension associated to the control distance associated to the sub-Laplacian
\[ 
    \mathcal{L}=-X_1^2-X_2^2-X_{3}^2-X_4^2-X_5^2-X_6^2,
\] can be computed from \eqref{Hausdorff-dimension} as follows.
\begin{align*}
    Q:&=\dim(H^1G)+ 2(\dim H^{2}G-\dim H^{1}G )
    = 6+2(8-6)=10.
\end{align*}

In the following Corollary we describe the $L^p$-$L^q$-boundedness of subelliptic pseudo-differential operators on $\textnormal{SU}(3).$ In this case we observe that one can identify $\widehat{\textnormal{SU}}(3)\cong \{D(p,q):p,q\in \mathbb{N}_0\}, $ where $D(p,q)$  in physical terms, $p$ is the number of quarks and $q$ is the number of antiquarks. The  construction of the unitary representations $D(p,q)$ can be found in \cite{Hall}. We keep in this case the standard notation  $\widehat{\textnormal{SU}}(3)$ by simplicity. 

\begin{corollary}\label{main:theorem:su:3} 
Let $1<p, q<\infty,$ and   $0\leq \delta< \rho\leq 1.$ Let us consider the H\"ormander sub-Laplacian $\mathcal{L}=-X_1^2-X_2^2-X_3^2-X_4^2-X_5^2-X_6^2.$  Then, the following statements hold.
\begin{itemize}
    \item Let $1<p\leq2\leq  q<\infty.$ Every pseudo-differential operator $$A\in \Psi^{m,\mathcal{L}}_{\rho,\delta}(\textnormal{SU}(3)\times \widehat{\textnormal{SU}}(3))$$ admits a bounded extension from $L^p(\textnormal{SU}(3))$ into $L^q(\textnormal{SU}(3)),$ that is
\begin{equation}\label{Lp-Lq:bound:su:3}
   \forall f\in C^{\infty}(\textnormal{SU}(3)),\, \Vert A f\Vert_{L^q(\textnormal{SU}(3))}\leq  C\Vert  f\Vert_{L^p(\textnormal{SU}(3))}\,\,\,
\end{equation} holds, if and only if, 
\begin{equation}\label{Necessary:condition:3:intro:su:3}
   m\leq -10\left(\frac{1}{p}- \frac{1}{q}\right).
\end{equation}  
\item Every pseudo-differential operator $A\in \Psi^{m,\mathcal{L}}_{\rho,\delta}(\textnormal{SU}(3)\times \widehat{\textnormal{SU}}(3))$ admits a bounded extension from $L^p(\textnormal{SU}(3))$ into $L^q(\textnormal{SU}(3)),$ that is \eqref{Lp-Lq:bound:su:3} holds, in the following cases: 
\begin{itemize}
        \item[(i)] if $1<p\leq q \leq 2$ and  \begin{equation}\label{Necessary:condition:4:intro:su:3}
   m\leq -10 \left( \frac{1}{p}-\frac{1}{q}+(1-\rho) \left(\frac{1}{q}-\frac{1}{2}\right)\right). \end{equation}    
\item[(ii)] if $2 \leq p \leq q<\infty$ and 

\begin{equation}\label{Necessary:condition:3:intro:su:3:3}
   m\leq -10\left( \frac{1}{p}-\frac{1}{q}+(1-\rho) \left(\frac{1}{2}-\frac{1}{p}\right)\right).
   \end{equation}
\end{itemize}
\end{itemize}
\end{corollary}

\subsection*{Conflict of interests statement.}   On behalf of all authors, the corresponding author states that there is no conflict of interest.

\subsection*{Data Availability Statements.}  Data sharing not applicable to this article as no datasets were generated or
analysed during the current study.\\

\bibliographystyle{amsplain}

\end{document}